\documentclass[a4paper,reqno,12pt]{amsart}
\usepackage{verbatim}
\usepackage{amssymb,amsmath,amsthm}
\usepackage{amsfonts}
\usepackage{array}

\newtheorem{theorem}{Theorem}
\newtheorem{lemma}{Lemma}

\theoremstyle{remark}

\newcommand{\leg}[2]{\left(\frac{#1}{#2}\right)}

\begin{document}

\setcounter{page}{1}

\title[]{A (human) proof of a triple binomial sum supercongruence}

\author{Roberto~Tauraso}
\address{Dipartimento di Matematica, 
Universit\`a di Roma ``Tor Vergata'', 
via della Ricerca Scientifica, 
00133 Roma, Italy}
\email{tauraso@mat.uniroma2.it}

\subjclass[2010]{11B65, 11A07.}


\begin{abstract} In a recent article, Apagodu and Zeilberger discuss some applications of an algorithm for finding and proving congruence identities (modulo primes)
of indefinite sums of many combinatorial sequence. At the end they propose some supercongruences as conjectures. Here we prove one of them and we leave some remarks for the others. 
\end{abstract}

\maketitle

\section{Introduction}
\noindent We will show that  
\begin{theorem} 
Let $p>2$ be a prime, and let $r$, $s$, $t$ be any positive integers, then
\begin{equation}\label{SS}
\sum_{m_1=0}^{rp-1}\sum_{m_2=0}^{sp-1}\sum_{m_3=0}^{tp-1}
\binom{m_1+m_2+m_3}{m_1,m_2,m_3}\equiv_{p^3} 
\sum_{m_1=0}^{r-1}\sum_{m_2=0}^{s-1}\sum_{m_3=0}^{t-1}
\binom{m_1+m_2+m_3}{m_1,m_2,m_3}.
\end{equation}
\end{theorem}

\noindent The above supercongruence appears as Conjecture 6' in \cite{AZ}.
After a preliminary section in which we collect some useful results,
we prove such conjecture in the third section. In the final section we provide
some remarks. 

\section{Preliminary results}
\noindent  For a $r>0$, let
${\bf s}=(s_1, \ldots, s_r)\in (\mathbb{Z^{\ast}})^r$ and let $x\in\mathbb{R}$. We define the multiple sum
$$
H_n({\bf s};x)=
\sum_{1\le k_1<\cdots<k_r\le n}
\prod_{i=1}^r\frac{x_i^{k_i}}{k_i^{|s_i|}}
\quad \mbox{with}\quad x_i = \left\{\begin{array}{lr}
        x & \mbox{if $s_i<0$, }\\
        1 & \mbox{if $s_i>0$. }
        \end{array}\right.$$
The number $l({\bf s}):=r$ is called the depth (or length) and $|{\bf s}|:=\sum_{j=1}^r|s_j|$ is the weight of the multiple sum.
By convention, these sums are zero if $n<r$. $H_n({\bf s};1)$ is the {\sl ordinary multiple harmonic sum} and in that case we will simply write $H_n({\bf s})$. 
Then it is known (see \cite[Sections 1 and 7]{TZ}) that for $p>3$,
\begin{align}
&H_{p-1}(1)\equiv_{p^2} 0\;,\label{C1}\\
&H_{p-1}(2)\equiv_{p} 0\;,\label{C2}\\
&H_{p-1}(1,1)\equiv_{p} 0\;,\label{C3}\\
&H_{p-1}(-1;2)\equiv_{p^2} -2q_p(2)\;,\label{C4}\\
&H_{p-1}(-1;1/2)\equiv_{p} q_p(2)\;,\label{C44}\\
&H_{p-1}(-2;-1)\equiv_{p} 0\;,\label{C66}\\
&H_{p-1}(-2;2)\equiv_{p} -q^2_p(2)\;,\label{C10}\\
&H_{p-1}(1,-1;-1)\equiv_{p} q^2_p(2)\;,\label{C55}\\
&H_{p-1}(1,-1;2)\equiv_{p} 0\;,\label{C5}\\
&H_{p-1}(-1,1;1/2)\equiv_{p} 0\;, \label{C6}
\end{align}
where $q_p(2)=(2^{p-1}-1)/p$. In the next section we will need the following results.

\begin{lemma} 
Let $p>2$ be a prime, then we have
\begin{equation}\label{B1}
\sum_{k=1}^{p-1}\frac{1}{k2^k}\sum_{j=1}^{k-1} \frac{2^j}{j}\equiv_{p} 0,
\end{equation}
and
\begin{equation}\label{B2}
\sum_{k=1}^{p-1}\frac{2^k}{k}\sum_{j=1}^{k-1} \frac{1}{j2^j}\equiv_{p}
-2q^2_p(2).
\end{equation}
\end{lemma} 
\begin{proof}
\noindent By \eqref{C6},
\begin{align*}
\sum_{k=1}^{p-1}\frac{1}{k2^k}\sum_{j=1}^{k-1} \frac{2^j}{j}
&\equiv_{p} \sum_{k=1}^{p-1}\frac{2^{-k}}{k}\sum_{j=1}^{k-1} \frac{2^{k-j}}{k-j}
=\sum_{k=1}^{p-1}\sum_{j=1}^{k-1} \frac{2^{-j}}{k(k-j)}\\
&=\sum_{j=1}^{p-2}\frac{2^{-j}}{j}\sum_{k=j+1}^{p-1}\left(\frac{1}{k-j}-\frac{1}{k}\right)
=\sum_{j=1}^{p-2}\frac{2^{-j}}{j}\sum_{k=1}^{p-j-1}\frac{1}{k}-H_{p-1}(-1,1;1/2)\\
&=\sum_{j=1}^{p-2}\frac{2^{-j}}{j}\sum_{k=j+1}^{p-1}\frac{1}{p-k}-H_{p-1}(-1,1;1/2)
\equiv_{p}-2H_{p-1}(-1,1;1/2)\equiv_{p}0.
\end{align*}
As regards \eqref{B2}, it follows from 
$$H_{p-1}(-1;2)\cdot H_{p-1}(-1;1/2)=H_{p-1}(2)+
\sum_{k=1}^{p-1}\frac{1}{k2^k}\sum_{j=1}^{k-1} \frac{2^j}{j}
+\sum_{k=1}^{p-1}\frac{2^k}{k}\sum_{j=1}^{k-1} \frac{1}{j2^j},$$
after using \eqref{C2}, \eqref{C4}, \eqref{C44}, and \eqref{B1}.
\end{proof}

\begin{lemma} 
Let $p>2$ be a prime, then we have
\begin{equation}\label{C7}
\sum_{k=2}^{p-1}\frac{1}{k^2}\sum_{m=k}^{p-1} \frac{(-1)^{m-k}}{\binom{m}{k}}\equiv_{p}
-2q_p(2).
\end{equation}
\end{lemma}
\begin{proof} We have that
\begin{align*}
\sum_{k=2}^{p-1}\frac{1}{k^2}\sum_{m=k}^{p-1} \frac{(-1)^{m-k}}{\binom{m}{k}}
&=\sum_{k=2}^{p-1}\frac{1}{k^2}\sum_{m=0}^{p-1-k} \frac{(-1)^{m}}{\binom{k+m}{m}}
=\sum_{k=1}^{p-2}\frac{1}{(p-k)^2}\sum_{m=0}^{k-1} \frac{(-1)^{m}}{\binom{p-k+m}{m}}\\
&\equiv_{p} \sum_{k=1}^{p-2}\frac{1}{k^2}\sum_{m=0}^{k-1} \frac{1}{\binom{k-1}{m}}
=\sum_{k=1}^{p-2}\frac{1}{k2^k}\sum_{j=1}^{k} \frac{2^j}{j}\\
&=\sum_{k=1}^{p-1}\frac{1}{k2^k}\sum_{j=1}^{k-1} \frac{2^j}{j}
+H_{p-1}(2)-\frac{H_{p-1}(-1;2)}{(p-1)2^{p-1}}\\
&\equiv_{p} 0+0 -2q_p(2)=-2q_p(2)
\end{align*}
where we used \eqref{C4}, \eqref{B1},
$$(-1)^m\binom{p-k+m}{m}=\frac{1}{m!}\prod_{j=k-m}^{k-1}(p-j)\equiv_{p}\binom{k-1}{m},$$
and the identity \cite[(2.4)]{G}
$$\sum_{m=0}^{k-1} \frac{1}{\binom{k-1}{m}}=
\frac{k}{2^k}\sum_{j=1}^{k} \frac{2^j}{j}.$$
\end{proof}

\begin{lemma} 
Let $i,j$ be non-negative integers. 

\noindent Let $p>2$ be a prime, then, for $0< r<p$, we have
\begin{equation}\label{CC22}
\binom{(i+j)p}{r+ip}\equiv_{p^3}
\binom{i+j}{i}\binom{p}{r}j\left(1-p\left((i+j-1)H_{r-1}(1)+\frac{i}{r}\right)\right),
\end{equation}
and
\begin{equation}\label{CC2}
\sum_{m=0}^{p-1}\binom{p-1+(i+j)p}{m+ip}\equiv_{p^3}
\binom{i+j}{i}\left(1+(i+j+1)pq_p(2)+\binom{i+j+1}{2} p^2q^2_p(2)\right).
\end{equation}
\end{lemma}
\begin{proof} We have that
\begin{align*}
\binom{(i+j)p}{r+ip}&=
\binom{(i+j)p}{ip}\frac{jp}{ip+r}
\prod_{k=1}^{r-1}\frac{jp-k}{ip+k}\\
&\equiv_{p^3}\binom{i+j}{i}\frac{jp(-1)^{j-1}}{ip+r}
\prod_{k=1}^{r-1}\frac{1-\frac{jp}{k}}{1+\frac{ip}{k}}\\
&\equiv_{p^3}\binom{i+j}{i}\frac{jp(-1)^{j-1}}{r}\left(1-\frac{ip}{r}\right)
\left(1-jpH_{r-1}(1)\right)\left(1-ipH_{r-1}(1)\right)\\
&\equiv_{p^3}\binom{i+j}{i}\frac{jp(-1)^{r-1}}{r}
\left(1-p\left((i+j)H_{r-1}(1)+\frac{i}{r}\right)\right).
\end{align*}
Congruence \eqref{CC22} follows as soon as we note that
$$\binom{p}{r}\equiv_{p^3}\frac{p(-1)^{r-1}}{r}
\left(1-pH_{r-1}(1)\right).$$
As regards \eqref{CC2},
\begin{align*}
\sum_{m=0}^{p-1}\binom{p-1+(i+j)p}{m+ip}&=
\sum_{m=0}^{p-1}\sum_{l=0}^{p-1}\binom{(i+j)p}{m-l+ip}\binom{p-1}{l}\\
&=\binom{(i+j)p}{ip}2^{p-1}+\sum_{l=1}^{p-1}\binom{p-1}{l}
\sum_{r=1}^{l}\left(\binom{(i+j)p}{r+ip}+\binom{(i+j)p}{r+jp}\right)\\
&\equiv_{p^3} \binom{i+j}{i}2^{p-1}
+\binom{i+j}{i}\sum_{l=1}^{p-1}\binom{p-1}{l}\sum_{r=1}^{l}\binom{p}{r}\\
&\qquad\qquad\qquad
\cdot \left((i+j)-p(i+j)(i+j-1)H_{r-1}(1)-\frac{2pij}{r}\right)
\end{align*}
where in the last step we applied \eqref{CC22}.
By \eqref{C3} and \eqref{C55},
\begin{align*}
\sum_{l=1}^{p-1}\binom{p-1}{l}\sum_{r=1}^{l}\binom{p}{r}H_{r-1}(1)
&\equiv_{p^2}-p\sum_{r=1}^{p-1}\frac{(-1)^rH_{r-1}(1)}{r}\sum_{r=l}^{p-1}(-1)^l\\
&=-\frac{p}{2}\left(H_{p-1}(1,-1;-1)+H_{p-1}(1,1)\right)
\equiv_{p^2} -\frac{pq_p^2(2)}{2},
\end{align*}
and similarly, by \eqref{C2} and \eqref{C66},
$$\sum_{l=1}^{p-1}\binom{p-1}{l}\sum_{r=1}^{l}\binom{p}{r}\frac{1}{r}
\equiv_{p^2}-\frac{p}{2}\left(H_{p-1}(-2;-1)+H_{p-1}(2)\right)
\equiv_{p^2} 0.$$
Moreover we have the identity
$$\sum_{l=1}^{p-1}\binom{p-1}{l}\sum_{r=1}^{l}\binom{p}{r}=2^{p-1}(2^{p-1}-1).$$
Finally, we obtain
$$\sum_{m=0}^{p-1}\binom{p-1+(i+j)p}{m+ip}
\equiv_{p^3}\binom{i+j}{i}\left(
2^{p-1}+(i+j)2^{p-1}(2^{p-1}-1)+p^2\binom{i+j}{2}q_p^2(2)
\right)$$
which is equivalent to \eqref{CC2}.
\end{proof}

\begin{lemma} Let $i,j$ be non-negative integers. 

\noindent Let $p>2$ be a prime, then for $0\leq m\leq k<p$, we have
\begin{equation}\label{CC1}
\binom{k+(i+j)p}{m+ip}\equiv_{p^2}
\binom{i+j}{i}\binom{k}{m}(1+p((i+j)H_k(1)-jH_{k-m}(1)-iH_m(1))),
\end{equation}
and
\begin{equation}\label{CC11}
\sum_{m=0}^k\binom{k+(i+j)p}{m+ip}\equiv_{p^2}
2^k\binom{i+j}{i}\left(
1+p(i+j)\sum_{m=1}^{k}\frac{1}{m2^m}\right).
\end{equation}
\end{lemma}
\begin{proof} 
We have that,
\begin{align*}
\binom{k+(i+j)p}{m+ip}&=
\sum_{l=0}^{k}\binom{(i+j)p}{m-l+ip}\binom{k}{l}\\
&=\binom{i+j}{j}\binom{k}{m}+
\sum_{l=1}^{m}\binom{(i+j)p}{l+ip}\binom{k}{m-l}
+\sum_{l=1}^{k-m}\binom{(i+j)p}{l+jp}\binom{k}{m+l}
\end{align*}
Then we apply \eqref{CC22} modulo $p^2$,
\begin{align*}
\binom{k+(i+j)p}{m+ip}&\equiv_{p^2}
\binom{i+j}{j}\left[\binom{k}{m}+
j\sum_{l=1}^{m}\binom{p}{l}\binom{k}{m-l}
+i\sum_{l=1}^{k-m}\binom{p}{l}\binom{k}{k-m-l}\right]\\
&=\binom{i+j}{j}\left[(1-j-i)\binom{k}{m}+
j\binom{p+k}{m}+i\binom{p+k}{k-m}\right]\\
&\equiv_{p^2}
\binom{i+j}{i}\binom{k}{m}(1+p((i+j)H_k(1)-jH_{k-m}(1)-iH_m(1))).
\end{align*}
From \eqref{CC1}, by summing over $m$ we obtain
$$\sum_{m=0}^k\binom{k+(i+j)p}{m+ip}\equiv_{p^2}
\binom{i+j}{i}\left(
2^k+p(i+j)\left(2^k H_k(1)-\sum_{m=0}^k\binom{k}{m}H_m(1)\right)
\right).$$
Congruence \eqref{CC11} follows after applying the identity \cite[(39)]{S} 
$$\sum_{m=0}^{k} \binom{k}{m}H_m(1)=2^k\left(H_k(1)-\sum_{m=1}^{k}\frac{1}{m2^m}\right).$$
\end{proof}

\section{Proof of the supercongruence \eqref{SS}.}

\noindent Let LHS and RHS be the left-hand side and the right-hand side of \eqref{SS}.
We have that
\begin{align*}
\mbox{LHS}&=
\sum_{m_1=0}^{rp-1}\sum_{m_2=0}^{sp-1} \binom{m_1+m_2}{m_1}
\sum_{m_3=0}^{tp-1}\binom{m_1+m_2+m_3}{m_3}\\
&=tp\sum_{m_1=0}^{rp-1}\sum_{m_2=0}^{sp-1} \binom{m_1+m_2}{m_1}
\binom{m_1+m_2+tp}{m_1+m_2}\frac{1}{m_1+m_2+1}\\
&=tp\sum_{m=0}^{rp-1}\sum_{k=m}^{m+sp-1} \binom{k}{m}
\binom{k+tp}{k}\frac{1}{k+1}\\
&=tp\sum_{i=0}^{r-1}\sum_{j=0}^{s-1}\sum_{m=0}^{p-1}\sum_{k=m}^{m+p-1} 
\binom{k+(i+j)p}{m+ip}\binom{k+(i+j+t)p}{k+(i+j)p}\frac{1}{k+(i+j)p+1}\\
&=\sum_{i=0}^{r-1}\sum_{j=0}^{s-1} (A_{ij}+B_{ij}+C_{ij})
\end{align*}
where
\begin{align*}
A_{ij}&:=tp\sum_{m=0}^{p-1}\sum_{k=m}^{p-1} 
\binom{k+(i+j)p}{m+ip}\binom{k+(i+j+t)p}{k+(i+j)p}\frac{1}{k+(i+j)p+1},\\
B_{ij}&:=\frac{tp}{(i+j+1)p+1}\binom{(i+j+t+1)p}{(i+j+1)p}
\sum_{m=1}^{p-1}\binom{(i+j+1)p}{m+ip},\\
C_{ij}&:=tp\sum_{m=0}^{p-1}\sum_{k=1}^{m-1} \binom{k+(i+j+1)p}{m+ip}
\binom{k+(i+j+t+1)p}{k+(i+j+1)p}\frac{1}{k+1+(i+j+1)p}.
\end{align*}
In a similar way
\begin{align*}
\mbox{RHS}
&=t\sum_{i=0}^{r-1}\sum_{j=0}^{s-1}
\binom{i+j}{i}\binom{i+j+t}{i+j}\frac{1}{i+j+1}.
\end{align*}
By Wolstenholme's theorem $\binom{ap}{bp}\equiv_{p^3} \binom{a}{b}$, and
$\binom{n_1p+n_0}{k_1p+k_0}\equiv_{p} \binom{n_1}{k_1}\binom{n_0}{k_0}$,
\begin{align*}
B_{ij}&\equiv_{p^3}\frac{t(i+j+1)p^2}{(i+j+1)p+1}\binom{i+j+t+1}{i+j+1}
\sum_{m=1}^{p-1}\binom{p-1+(i+j)p}{m-1+ip}\frac{1}{m+ip}\\
&\equiv_{p^3}tp^2(i+j+1)\binom{i+j+t+1}{i+j+1}\binom{i+j}{i}
\sum_{m=1}^{p-1}\binom{p-1}{m-1}\frac{1}{m}\\
&\equiv_{p^3}2tp^2(i+j+t+1)\binom{i+j+t}{i,j,t}q_p(2),
\end{align*}
where
$$\sum_{m=1}^{p-1}\binom{p-1}{m-1}\frac{1}{m}=
\frac{1}{p}\sum_{m=1}^{p-1}\binom{p}{m}=\frac{2^p-2}{p}=2q_p(2).$$
Finally, we note that for $k<m<p$
$$\binom{k+p}{m}
=\frac{1}{m!}\prod_{j=1}^{k}(j+p)\cdot p\cdot 
\prod_{j=1}^{m-(k+1)}(p-j)\equiv_{p^2} \frac{p(-1)^{m-(k+1)}}{(k+1)\binom{m}{k+1}},$$
and
\begin{align*}
\binom{k+(i+j+1)p}{m+ip}
&=\frac{(k+p+(i+j)p)\cdots(p+(i+j)p)}{(m+ip)\cdots(1+ip)}
\binom{k+p-m+(i+j)p}{ip}\\
&\equiv_{p^2}(i+j+1)\binom{k+p}{m}\binom{k-m+p+(i+j)p}{ip}\\
&\equiv_{p^2}(i+j+1)\frac{p(-1)^{m-(k+1)}}{(k+1)\binom{m}{k+1}}\binom{i+j}{i}.
\end{align*}
Moreover
$$\binom{k+(i+j+t+1)p}{k+(i+j+1)p}\equiv_{p} \binom{i+j+t+1}{i+j+1},$$
and we obtain
\begin{align*}
C_{ij}&=tp^2(i+j+1)\binom{i+j+t+1}{i+j+1}\binom{i+j}{i}
\sum_{m=0}^{p-1}\sum_{k=1}^{m-1}\frac{(-1)^{m-(k+1)}}{(k+1)^2\binom{m}{k+1}}\\
&\equiv_{p^3}tp^2(i+j+t+1)\binom{i+j+t}{i,j,t}\sum_{k=1}^{p-2}\frac{1}{(k+1)^2}\sum_{m=k+1}^{p-1} \frac{(-1)^{m-(k+1)}}{\binom{m}{k+1}}\\
&\equiv_{p^3}tp^2(i+j+t+1)\binom{i+j+t}{i,j,t}\sum_{k=2}^{p-1}\frac{1}{k^2}\sum_{m=k}^{p-1} \frac{(-1)^{m-k}}{\binom{m}{k}}\\
&\equiv_{p^3} -2tp^2(i+j+t+1)\binom{i+j+t}{i,j,t}q_p(2)
\end{align*}
where in the last step we used \eqref{C7}.
Therefore $B_{ij}+C_{ij}\equiv_{p^3} 0$ and it suffices to show that
\begin{equation}\label{ff}
A_{ij}\equiv_{p^3}\frac{t}{i+j+1}\binom{i+j}{i}\binom{i+j+t}{i+j}.
\end{equation}
Now, by \eqref{CC2},
\begin{align*}
p\sum_{k=p-1}^{p-1} \sum_{m=0}^{k}\cdots
&=\frac{1}{i+j+t+1}\binom{(i+j+t+1)p}{(i+j+1)p}\sum_{m=0}^{k}\binom{p-1+(i+j)p}{m+ip}\\
&\equiv_{p^3}
\frac{1}{i+j+1}\binom{i+j+t}{i+j}\binom{i+j}{i}2^{(p-1)(i+j+1)}\\
&\equiv_{p^3}
\binom{i+j+t}{i,j,t}\left(\frac{1}{i+j+1}+pq_p(2)+\frac{(i+j)}{2}\,p^2q^2_p(2)\right).
\end{align*}
because
$$\frac{(2^{p-1})^{(i+j+1)}}{i+j+1}=\frac{(1+pq_p(2))^{i+j+1}}{i+j+1}
\equiv_{p^3} \frac{1}{i+j+1}+pq_p(2)+\frac{(i+j)}{2}\,p^2q^2_p(2).$$
Moreover by \eqref{CC1} and \eqref{CC11},
\begin{align*}
p\sum_{k=0}^{p-2} \sum_{m=0}^{k}\cdots
&\equiv_{p^3}
p\sum_{k=0}^{p-2}\left(\frac{1}{k+1}-\frac{p(i+j)}{(k+1)^2}\right)
\binom{i+j+t}{i+j}(1+p(i+j)H_k(1))\\
   &\qquad \cdot2^k\binom{i+j}{i}\left(
1+p(i+j)\sum_{m=1}^{k}\frac{1}{m2^m}\right)\\
&\equiv_{p^3}
\frac{p}{2}
\binom{i+j+t}{i,j,t}\Bigg(H_{p-1}(-1;2)-p(i+j)H_{p-1}(-2;2)\\
&\qquad\left. +p(i+j)H_{p-1}(1,-1;2) 
+p(i+j)\sum_{k=0}^{p-2}\frac{1}{k+1}\sum_{m=1}^{k}\frac{1}{m2^m}\right)\\
&\equiv_{p^3}
\binom{i+j+t}{i,j,t}\left(-pq_p(2)-\frac{(i+j)}{2}p^2q_p^2(2)\right),
\end{align*}   
where in the last step we used \eqref{C4}, \eqref{C10}, \eqref{C5}, and \eqref{B2}.
Hence
\begin{align*}
A_{ij}=tp\sum_{k=0}^{p-1} \sum_{m=0}^{k}\cdots
&\equiv_{p^3} \frac{t}{i+j+1}\binom{i+j+t}{i,j,t}.
\end{align*}   
and the proof of \eqref{ff} is complete.

\section{Some further remarks}
In \cite{AZ} appeared some other supercongruences. 

\begin{itemize}

\item[-]\noindent  Supercongruence 1: for any prime $p$,
$$\sum_{n=0}^{p-1} \binom{2n}{n}\equiv_{p^2}\leg{p}{3}$$
which is congruence (1.9) at p. 647 in \cite{ST} (here $\leg{p}{3}$ is the Legendre symbol).

\item[-]\noindent  Supercongruence 2: for any prime $p$,
$$\sum_{n=0}^{p-1} C_n\equiv_{p^2}\frac{1}{2}\left(3\leg{p}{3}-1\right)$$
which is congruence (1.7) at p. 647 in \cite{ST} (here $C_n$ is the $n$th Catalan numbers).

\item[-]\noindent  Supercongruence 5: for any prime $p$,
$$\sum_{n=0}^{p-1}\sum_{m=0}^{p-1} \binom{n+m}{m}^2\equiv_{p^2}\leg{p}{3}$$
which is equivalent to Supercongruence 1 because
\begin{align*}
\sum_{n=0}^{p-1}\sum_{m=0}^{p-1} \binom{n+m}{m}^2&=
\sum_{k=0}^{p-1}\sum_{m=0}^{k} \binom{k}{m}^2+\sum_{k=0}^{p-2}\sum_{m=k+1}^{p-1} \binom{k+p}{m}^2
\equiv_{p^2} \sum_{k=0}^{p-1}\binom{2k}{k}
\end{align*}
where we used the fact that $p$ divides $\binom{k+p}{m}$ and by Vandermonde's convolution  
$\sum_{m=0}^{k} \binom{k}{m}^2=\binom{2k}{k}$.
\end{itemize}

\medskip


\begin{thebibliography}{99}

\bibitem{AZ} M. Apagodu and D. Zeilberger,
\emph{Using the ``Freshman's Dream'' to Prove Combinatorial Congruences},
arXiv:1606.03351 (june 2016).


\bibitem{G} H. W. Gould,
\emph{Combinatorial Identities}, Morgantown W. Va (1972).

\bibitem{S} 
P. Paule and C. Schneider, 
\emph{Computer proofs of a new family of harmonic number identities}, 
Adv. Appl. Math., \textbf{31} (2003), 359-378.


\bibitem{ST} 
Z.-W. Sun and R. Tauraso, 
\emph{On some new congruences for binomial coefficients}, 
Int. J. Number Theory , \textbf{7} (2011), 645-662.

\bibitem{TZ} 
R. Tauraso and J. Zhao, 
\emph{Congruences of alternating multiple harmonic sums}, 
J. Comb. Number Theory, \textbf{2} (2010), 129-159.
\end{thebibliography}
\end{document}